\documentclass{amsart}

\usepackage[utf8]{inputenc}
\usepackage{enumerate}
\usepackage{xcolor,pgf,tikz,pgflibraryarrows,pgffor,pgflibrarysnakes}
\usepackage{float}
\usepackage{enumitem}
\setlist[enumerate]{leftmargin=*, nosep}

\usepackage[colorlinks=true,linkcolor=blue]{hyperref}

\newtheorem{theorem}{Theorem}[section]

\newtheorem{corollary}[theorem]{Corollary}
\newtheorem{prop}[theorem]{Proposition}

\newtheorem{pbm}[theorem]{Problem}
\newtheorem{question}[theorem]{Question}

\theoremstyle{definition}

\theoremstyle{remark}
\newtheorem{remark}[theorem]{Remark}

\usepackage[flushmargin]{footmisc}
\newcommand\blfootnote[1]{%
	\begingroup
	\renewcommand\thefootnote{}\footnote{#1}%
	\addtocounter{footnote}{-1}%
	\endgroup}

\newcommand{\NN}{\mathbb{N}}
\newcommand{\ZZ}{\mathbb{Z}}
\newcommand{\KK}{\mathbb{K}}
\newcommand{\DD}{\mathbb{D}}

\newcommand{\cl}{\overline}
\newcommand{\lspan}{\operatorname{span}}
\newcommand{\eps}{\varepsilon}
\newcommand{\val}{\operatorname{val}}

\begin{document}

\title[Expansive operators on Fr\'echet spaces without continuous norm]{Expansive operators on Fr\'echet\\ spaces without continuous norm}
\author{Antoni L\'opez-Mart\'inez \& Quentin Menet}

\maketitle

\begin{abstract}
	We investigate the existence of (metrically) expansive operators on Fr\'echet spaces without continuous norm. We first solve a problem posed by Bernardes, Caraballo, Darji, F\'avaro, and Peris in 2025, by showing that there is no expansive operator on the Fr\'echet space $\omega$ of all sequences. We prove this fact by investigating the existence of orbits remaining cofinitely inside some fixed closed subspace of finite codimension. We then give examples of expansive operators on some other Fr\'echet spaces without continuous norm.
\end{abstract}




\section{Introduction}

Given a continuous linear operator $T : X \longrightarrow X$ acting on a topological vector space $X$ (from now on simply called an \textit{operator}), the field of linear dynamics studies the long-term behaviour of $T$ as a dynamical system, connecting the branches of functional analysis, operator theory, and topological dynamics. For a long time the research in this area focused mainly on hypercyclicity (existence of a dense orbit) and related concepts such as Devaney chaos or frequent hypercyclicity. This focus on dense orbits was largely motivated by the famous \textit{invariant subspace problem} and its \textit{invariant subset problem} version (see the monographs \cite{BayartMa2009_book_dynamics,GrossePe2011_book_linear}).

Recently, research in linear dynamics has expanded to include notions originating in smooth and topological dynamics. Rather than only asking whether dense orbits exist, researchers now also investigate properties such as hyperbolicity, structural stability, the shadowing property, chain recurrence, and \textit{expansivity}, usually in the context of invertible operators. In topological dynamics, expansivity captures the idea that distinct points are eventually pushed apart by iteration. In the setting of Banach spaces, the study of expansivity for continuous linear operators has revealed strong connections with stability, classical dynamical properties and spectral theory:
\begin{enumerate}[label=$\bullet$]
	\item \textbf{Hyperbolicity, stability, and shadowing.} For an invertible operator acting on a Banach space, being hyperbolic is equivalent to having both \textit{expansivity} and the shadowing property \cite{BernardesMe2021_ETDS_shadowing}. Moreover, if an invertible operator is structurally stable and presents positive expansivity, then it is hyperbolic \cite{BernardesMe2021_ETDS_shadowing}.
	
	\item \textbf{Hypercyclicity and chaos.} Expansivity imposes severe restrictions on the existence of dense orbits: a \textit{uniformly positively expansive} operator on a Banach space cannot be hypercyclic, and therefore cannot be Devaney chaotic \cite{BernardesCiDarMePu2018_JMAA_expansivity}.
	
	\item \textbf{Spectral characterizations.} On a Banach space $X$, an invertible operator $T$ is \textit{uniformly positively expansive} if and only if its approximate point spectrum does not intersect the closed unit disk, that is, if $\sigma_a(T) \cap \cl{\DD} = \emptyset$ \cite{BernardesCiDarMePu2018_JMAA_expansivity}.
\end{enumerate}

These ``new'' properties (hyperbolicity, stability, shadowing, expansivity) have been extended beyond Banach spaces to operators acting on locally convex spaces, and specifically on Fr\'echet spaces, in \cite{BernardesCaDarFaPe2025_JFA_generalized}. However, once we move outside the Banach framework there appear some technical and conceptual differences. For instance, although the operator two times the identity $T := 2 I_X$ is expansive on each Banach space $X$, the following question was stated in~\cite{BernardesCaDarFaPe2025_JFA_generalized}:

\begin{question}[\textbf{\cite[Problem~B]{BernardesCaDarFaPe2025_JFA_generalized}}]\label{Ques:expansive}
	Does every Fr\'echet space (or locally convex space) support an expansive operator? What about the Fr\'echet space $\omega$?
\end{question}

In \cite{BernardesCaDarFaPe2025_JFA_generalized} it was emphasized that here \textit{expansive} is meant in the metric sense: a homeomorphism $T : X \longrightarrow X$ on a metric space $(X,d)$ is said to be \textit{expansive} with respect to the metric $d$ if there is some $c>0$, called expansivity constant for~$T$, such that for every pair $x, y \in X$ with $x \neq y$ there exists some $n \in \ZZ$ for which $d(T^nx,T^ny) \geq c$. When $T : X \longrightarrow X$ is a continuous but not necessarily invertible map, then $T$ is said to be \textit{positively expansive} with respect to the metric $d$ if there is some $c>0$, called positive expansivity constant for~$T$, such that for every pair $x, y \in X$ with $x \neq y$ there exists some $n \in \NN_0$ for which $d(T^nx,T^ny) \geq c$.

Given a Fr\'echet space $X$ whose topology is induced by an increasing sequence of seminorms $(p_k)_{k \geq 1}$, a standard way of defining a metric that induces its topology is using the formula
\begin{equation}\label{eq:metric}
	d(x,y) := \sum_{k=1}^{\infty} \frac{\min\{ 1 , p_k(x-y) \}}{2^k} \quad \text{ for all } x,y \in X.
\end{equation}
Although other translation-invariant metrics can induce the same topology on $X$, they are all uniformly equivalent to $d$ by translation-invariance. Since expansivity is preserved under uniformly equivalent metrics, the choice of the defining metric and, in particular, the choice of the sequence of seminorms $(p_k)_{k \geq 1}$, does not affect whether an operator $T$ on $X$ is expansive or positively expansive. Therefore, fixed any sequence of seminorms defining the topology of $X$, we can study Question~\ref{Ques:expansive} with respect to the canonical metric \eqref{eq:metric} without loss of generality.

Moreover, given any translation-invariant metric $d$ inducing the topology of a metrizable topological vector space $X$, using linearity it is clear that an operator $T$ is (positively) expansive on $X$ if and only if there is some $c>0$ such that for every $x \in X \setminus \{0_X\}$ there exists $n \in \ZZ$ ($n \in \NN_0$) satisfying $d(T^nx,0_X) \geq c$. Using this characterization and the metric \eqref{eq:metric} it follows that, if a Fr\'echet space $X$ admits a continuous norm, then $T := 2 I_X$ is positively expansive (and hence expansive): if $p_{k_0}$ is a continuous norm on $X$, then for any $x \in X \setminus \{0_X\}$ we have that
\[
d(T^nx,0_X) = d(2^nx,0_X) \geq 2^{-k_0} \min\{ 1 , 2^n p_{k_0}(x) \} = 2^{-k_0}
\]
for all sufficiently large $n \in \NN_0$. Thus, $c := 2^{-k_0}$ serves as a positive expansivity constant for $2 I_X$. Consequently, Question~\ref{Ques:expansive} can be reformulated as follows: \textit{Does every Fr\'echet space without continuous norm support an expansive operator?}

In the second section of this paper we show that there is no expansive or positively expansive operator on the Fr\'echet space of all sequences $\omega$, answering Question~\ref{Ques:expansive} in the negative. This result has also been obtained recently and independently by Wu and Wang \cite{WuWang2026_arXiv_shadowing} by working with a completely different approach relying on the notion of modules. Our method differs from that used in \cite{WuWang2026_arXiv_shadowing} by focusing on the existence of orbits remaining cofinitely inside some fixed closed subspace of finite codimension in $\omega$. This technique may have its own interest.

In the third section we then establish the existence of expansive operators for several Fr\'echet spaces without continuous norm. In particular, we establish some sufficient conditions for a Banach space $X$ to fulfill that the direct sum $X \oplus \omega$ does admit an invertible positively expansive and hence expansive operator. We end the paper showing that some infinite $\omega$-products of infinite-dimensional Banach spaces also support invertible positively expansive and hence expansive operators.

\section{Non-existence of expansive operators on $\omega$}

We recall that $\omega = \KK^{\NN}$ is the Fr\'echet space formed by all the real or complex sequences endowed with the topology of the coordinate convergence. Such topology is induced by the increasing sequence of seminorms $(p_k)_{k \geq 1}$, defined for each positive integer $k \in \NN$ as
\begin{equation}\label{eq:omega.seminorms}
	p_k(x) := \max_{1 \leq j \leq k} \left| [x]_j \right| \quad \text{ for each } x=([x]_j)_{j \geq 1} \in \omega.
\end{equation}
We also recall that the \textit{valuation} of a non-zero vector $x \in \omega \setminus \{0_{\omega}\}$, from now on denoted by $\val(x)$, is defined as the smallest index $j \geq 1$ such that $[x]_j \neq 0$. We let $\val(0_{\omega})=\infty$. We can then define the valuation of a subspace $M \subset \omega$ as follows:
\[
\val(M) := \min\{ \val(x) : x \in M \}.
\]
In particular, $\val(M) = \infty$ if and only if $M = \{0_{\omega}\}$. Moreover, using that the topology of $\omega$ only depends on the coordinate convergence, the reader should notice and we will repeatedly use the following fact:
\begin{enumerate}[label=$\bullet$]
	\item \textit{every sequence $(f_n)_{n \geq 1}$ of vectors in $\omega$ such that its sequence $(\val(f_n))_{n \geq 1}$ tends to infinity fulfills that the series $\sum_{n=1}^{\infty} f_n$ is convergent in $\omega$.}
\end{enumerate}
With these notations, and recalling that our objective is proving that there is no expansive or positively expansive operator on $\omega$, our approach will be based on the following three observations:
\begin{enumerate}[label=\arabic*.]
	\item Given the metric $d$ on $\omega$ defined by \eqref{eq:metric} with respect to the the seminorms \eqref{eq:omega.seminorms}, it is easy to notice that if $\val(x) > k$ then $d(x,0_{\omega}) \leq 2^{-k}$.
	
	\item Thus, to guarantee the existence of a non-zero vector $x \in \omega \setminus \{0_{\omega}\}$ fulfilling that $d(T^nx,0_{\omega}) \leq 2^{-k}$ for all $n \in \ZZ$ (or all $n \in \NN$), it suffices to show that there exists a vector $x$ such that $\val(T^nx) > k$ for all $n \in \ZZ$ (or all $n \in \NN$).
	
	\item Since the set $\ker(p_k) = \{ x \in \omega : p_k(x) = 0 \}$ of sequences with valuation bigger than a certain positive integer $k \in \NN$ is a closed subspace of finite codimension in~$\omega$, it is then natural to investigate when the orbit of a vector remains cofinitely in a given closed subspace of finite codimension.
\end{enumerate}
The following proposition will help us to deal with this study:

\begin{prop}\label{Pro:F_n+E_n}
	Let $M$ be a closed subspace of $\omega$. Then, for every non-increasing sequence $(E_n)_{n \geq 1}$ of closed subspaces of finite codimension in $M$, there exists a sequence of finite-dimensional subspaces $(F_n)_{n \geq 1}$ such that:
	\begin{enumerate}[label=$\bullet$]
		\item $M=F_1\oplus E_1$ and $E_n=F_{n+1}\oplus E_{n+1}$ for every $n \geq 1$;
		
		\item the sequence $(\val(F_n))_{n \geq 1}$ tends to infinity and hence every series $\sum_{n=1}^{\infty} f_n$ is convergent whenever $(f_n)_{n \geq 1}$ is a sequence with $f_n \in F_n$ for each $n \geq 1$;
		
		\item for every $x \in M$ there exist a unique sequence $(f_n)_{n \geq 1}$, with $f_n\in F_n$ for each $n \geq 1$, and a unique vector $e \in \bigcap_{n \geq 1} E_n$ such that $x = \sum_{n=1}^{\infty} f_n + e$.
\end{enumerate}
\end{prop}
\begin{proof}
	We let $E_0=M$. Fixed any $n \geq 0$, the assumptions imply that $E_{n+1}$ is a closed subspace of finite codimension in $E_n$, so that for each $n \geq 0$ there exists a finite dimensional subspace $F_{n+1}$ fulfilling that $E_n = F_{n+1} \oplus E_{n+1}$.
	
	We \textit{remark} that the subspaces $(F_n)_{n \geq 1}$ can be chosen so that its sequence of valuations $(\val(F_n))_{n \geq 1}$ tends to infinity and thus such that every series $\sum_{n=0}^{\infty} f_n$ converges whenever $f_n \in F_n$ for each $n \geq 1$. Indeed, if we consider $k \geq 1$ and $n_k \geq 1$ and we assume that $\val(F_n) \geq k$ for every $n \geq n_k$, then we can show that there exist $n_{k+1} > n_k$ and a choice of $(F_n)_{n \geq n_{k+1}}$ such that $\val(F_n) \geq k+1$ for every $n \geq n_{k+1}$. There are in fact two possible cases:
	\begin{enumerate}[label=\arabic*.]
		\item There exists some $n_{k+1} > n_k$ such that $\val(F_{n}) \geq k+1$ for every $n \geq n_{k+1}$. In this case we have nothing to do.
		
		\item We have that $\val(F_n) = k$ for infinitely many integers $n > n_k$. In this case we enumerate these indices by $(l_j)_{j \geq 1}$ and, recursively, we proceed as follows:
		\begin{enumerate}[label=$\bullet$]
			\item Consider a basis of $F_{l_1}$ containing only one vector $f_{l_1}$ of valuation equal to $k$.
			
			\item Using that $\val(F_{l_2})=k$ and that $F_{l_2} \subset E_{l_1}$, we can find a vector $e_{l_1} \in E_{l_1}$ with valuation equal to $k$.
			
			\item Let $F_{j_1}'$ be the finite dimensional subspace spanned by the basis of $F_{j_1}$ but replacing the vector $f_{l_1}$ by a linear combination of $f_{l_1}$ and $e_{l_1}$ such that its valuation is strictly bigger than $k$. Thus, $\val(F_{l_1}')>k$ and $E_{l_1-1} = F_{l_1}' \oplus E_{l_1}$.
		\end{enumerate}
		Repeating this argument for each $F_{l_j}$ with $j \geq 2$ we obtain a new sequence $(F_{l_j}')_{j \geq 1}$ of finite dimensional subspaces with $\val(F_{l_j}') > k$ and $E_{l_{j-1}} = F_{l_j}' \oplus E_{l_j}$ for all $j \geq 1$. Renaming these spaces by $(F_{l_j})_{j \geq 1}$ and letting $n_{k+1} := n_k+1$ we finally get that $\val(F_n) \geq k+1$ for all $n \geq n_{k+1}$ and that we still have the equality $E_n = F_{n+1} \oplus E_{n+1}$ for all $n \geq 0$, completing the \textit{remark}.\\
	\end{enumerate}
	
	Finally, pick any $x \in M$ and let us show that there exist a unique sequence $(f_n)_{n \geq 1}$, with $f_n \in F_n$ for each $n \in \NN$, and a unique vector $e \in \bigcap_{n \geq 1} E_n$ such that $x = \sum_{n=1}^{\infty} f_n + e$. Indeed, by using the fact that $M = F_1 \oplus E_1$ we get unique vectors $f_1 \in F_1$ and $e_1 \in E_1$ such that $x = f_1 + e_1$. Then, since $E_n=F_{n+1}\oplus E_{n+1}$ for every $n \geq 1$, we know that for each $e_n \in E_n$ there exist a unique $f_{n+1} \in F_{n+1}$ and a unique $e_{n+1} \in E_{n+1}$ such that $e_n = f_{n+1} + e_{n+1}$. We thus obtain sequences $(f_n)_{n \geq 1}$ and $(e_n)_{n \geq 1}$, with $f_n \in F_n$ and $e_n \in E_n$ for each $n \in \NN$, such that
	\[
	x = \sum_{k=1}^n f_k + e_n \quad \text{ for every } n \geq 1.
	\]
	Since the series $\sum_{k=1}^{\infty}f_k$ converges, we conclude that the sequence $(e_n)_{n \geq 1}$ is also convergent to some vector $e \in \omega$ and that
	\[
	x = \sum_{k=1}^{\infty} f_k + e.
	\]
	Finally, since $(e_n)_{n \geq 1}$ is a convergent sequence and $(E_n)_{n \geq 1}$ is a non-increasing sequence of closed subspaces, we get that $e \in  \bigcap_{n \geq 1} E_n$. Let us remark that we have unicity of both $(f_n)_{n \geq 1}$ and $e$ since, for each $n \geq 1$, the expression
	\[
	x = \sum_{k=1}^{n}f_k + \left( \sum_{k=n+1}^{\infty} f_k + e \right)
	\]
	has to be the unique decomposition of $x$ in the direct sum $\left( \oplus_{k=1}^n F_k \right) \oplus E_n$.
\end{proof}

For an operator $T$ on $\omega$ and a closed subspace of finite codimension $E \subset \omega$, we can use Proposition~\ref{Pro:F_n+E_n} on the non-increasing sequence of finite-codimensional subspaces $E_n = \bigcap_{k=0}^{n-1} T^{-k}(E)$ to establish the following result:

\begin{theorem}\label{The:intersect-n>=1}
	For every operator $T$ on $\omega$ and every closed subspace of finite codimension $E \subset \omega$ we have that
	\[
	\bigcap_{n \geq 1} T^{-n}(E) \neq \{0_{\omega}\}.
	\]
\end{theorem}
\begin{remark}\label{Rem:injectivity}
	Before proving Theorem~\ref{The:intersect-n>=1} let us remark that, since 
	\[
	\bigcap_{n \geq 1}T^{-n}(E) = T^{-1}\left(\bigcap_{n \geq 0}T^{-n}(E)\right),
	\]
	it is straightforward to verify that
	\[
	\bigcap_{n \geq 1}T^{-n}(E) \neq \{0_{\omega}\} \quad \Leftrightarrow \quad \text{$T$ is not injective \ or \ } \bigcap_{n \geq 0}T^{-n}(E) \neq \{0_{\omega}\}.
	\]
\end{remark}
\begin{proof}[Proof of Theorem \ref{The:intersect-n>=1}]
	Let $T$ be an operator on $\omega$ and let $E$ be a closed subspace of finite codimension in $\omega$. If $T$ is not injective then we directly get that
	\[
	\{0_{\omega}\} \neq \ker(T) \subset \bigcap_{n \geq 1}T^{-n}(E).
	\]
	Thus, assume that $T$ is injective and notice that, by Remark~\ref{Rem:injectivity}, to complete the proof we must show that $\bigcap_{n \geq 0} T^{-n}(E) \neq \{0_{\omega}\}$. To this end we let $E_0 := \omega$ and
	\[
	E_n := E \cap T^{-1}(E) \cap \cdots \cap T^{-(n-1)}(E) \quad \text{ for each } n \geq 1,
	\]
	so that $(E_n)_{n \geq 1}$ is indeed a non-increasing sequence of closed subspaces of finite codimension in $\omega$. Using Proposition~\ref{Pro:F_n+E_n} we can find a sequence of finite-dimensional subspaces $(F_n)_{n \geq 1}$ such that:
	\begin{enumerate}[label=$\bullet$]
		\item $\omega=F_1\oplus E_1$ and $E_n=F_{n+1}\oplus E_{n+1}$ for every $n \geq 1$;
		
		\item the sequence $(\val(F_n))_{n \geq 1}$ tends to infinity and hence every series $\sum_{n=1}^{\infty} f_n$ is convergent whenever $(f_n)_{n \geq 1}$ is a sequence with $f_n \in F_n$ for each $n \geq 1$;
		
		\item for every $x \in \omega$ there exist a unique sequence $(f_n)_{n \geq 1}$, with $f_n\in F_n$ for each $n \geq 1$, and a unique vector $e \in \bigcap_{n \geq 1} E_n$ such that $x = \sum_{n=1}^{\infty} f_n + e$.\\
	\end{enumerate}
	
	We start noticing that, if there exists $n \geq 0$ such that $E_n=E_{n+1}$, then $E_k = E_n$ for all $k \geq n$. Indeed, if $E_n = E_{n+1}$ and if we consider $k \geq n$ and $x \in E_k$ then we have that $T^{k-n}x \in E_n = E_{n+1}$ and thus we have that $T^nT^{k-n}x = T^kx \in E$. It follows that $x \in E_{k+1}$ and thus that $E_k = E_{k+1}$ for all $k \geq n$. Therefore, if there exists $n \geq 0$ such that $E_n = E_{n+1}$ we would have that
	\[
	\bigcap_{k \geq 0} T^{-k}(E) = E_n \neq \{0_{\omega}\}.
	\]
	Hence, assume that $E_{n+1}$ is a strict subspace of $E_n$ for all $n \geq 0$, which implies that $F_{n+1} \neq \{0_{\omega}\}$ for all $n \geq 0$. Let us check that there exists $N \geq 1$ such that
	\begin{equation}\label{eq:dimF_N}
		\dim(F_n) = \dim(F_N) \quad \text{ for all } n \geq N.
	\end{equation}
	To show this fact notice that given any $n \geq 0$, since $E_{n+1} = F_{n+2}\oplus E_{n+2}$, we have that each $x \in F_{n+2} \setminus \{0_{\omega}\}$ belongs to $E_{n+1}$ but does not belong to $E_{n+2}$, implying that $Tx \in E_n \setminus E_{n+1}$ and hence that $T(F_{n+2}) \subset E_n$ while $T(F_{n+2}) \cap E_{n+1} = \{0_{\omega}\}$. Since $E_n = F_{n+1} \oplus E_{n+1}$, it follows that $\dim(T(F_{n+2})) \leq \dim(F_{n+1})$. Moreover, since we are assuming that $T$ is injective, we deduce that
	\[
	\dim(F_{n+2}) = \dim(T(F_{n+2})) \leq \dim(F_{n+1}).
	\]
	This argument shows that $(\dim(F_n))_{n \geq 1}$ is indeed a decreasing sequence of positive integers, so that there exists some $N \geq 1$ fulfilling \eqref{eq:dimF_N}.
	
	Now we claim that, given any $n \geq N$ and any $x \in E_n$, there exists $(f_k)_{k \geq n+2}$, with $f_k \in F_k$ for each $k \geq n+2$, and a vector $e \in \bigcap_{k \geq 0} T^{-k}(E)$ such that
	\begin{equation}\label{eq:x=T(sum)+e}
		x = T\left( \sum_{k=n+2}^{\infty} f_k \right) + e.
	\end{equation}
	Actually, since $x \in E_n$ we have that $x = f_{n+1} + e_{n+1}$ for some $f_{n+1} \in F_{n+1}$ and some $e_{n+1} \in E_{n+1}$. Then, using that $T(F_{n+2}) \subset E_n =  F_{n+1} \oplus E_{n+1}$, that $T(F_{n+2}) \cap E_{n+1} = \{0_{\omega}\}$ and that $\dim(T(F_{n+2})) = \dim(F_{n+2}) = \dim(F_{n+1})$, there exists some $f_{n+2} \in F_{n+2}$ such that $Tf_{n+2}-f_{n+1} \in E_{n+1}$. If this was not the case, then the finite-dimensional space $F_{n+1}' := T(F_{n+2}) \oplus \lspan(f_{n+1})$ would fulfill that $\dim(F_{n+1}') = \dim(T(F_{n+2}))+1 > \dim(F_{n+1})$ while $F_{n+1}' \cap E_{n+1} = \{0_{\omega}\}$, which contradicts that $\dim(F_{n+1})$ is the codimension of $E_{n+1}$ as a subspace of $E_n$. It follows that $x=Tf_{n+2}+e'_{n+1}$ for some $e'_{n+1}\in E_{n+1}$. Repeating the argument for the vector $e_{n+1}' \in E_{n+1}$ and continuing recursively we get a sequence $(f_k)_{k \geq n+2}$, with $f_k \in F_k$ for each $k \geq n+2$, and a sequence $(e'_k)_{k \geq n+1}$, with $e'_k \in E_k$ for each $k \geq n+1$, such that for every $K \geq n+2$ we have the equality
	\[
	x = \sum_{k=n+2}^{K} Tf_k + e'_{K-1} = T\left( \sum_{k=n+2}^{K} f_k \right) + e'_{K-1} .
	\]
	Since $\sum_{k=n+2}^{\infty}f_{k}$ is convergent by Proposition~\ref{Pro:F_n+E_n}, we finally obtain \eqref{eq:x=T(sum)+e}, where
	\[
	e = \lim_{k \to \infty} e'_k \in \bigcap_{k \geq 0}T^{-k}(E).
	\]
	In particular, \eqref{eq:x=T(sum)+e} implies that $E_n \subset T(E_{n+1}) \oplus \bigcap_{k \geq 0}T^{-k}(E)$ for every $n \geq N$.
	
	Finally, by contradiction, assume that $\bigcap_{k \geq 0}T^{-k}(E)=\{0_{\omega}\}$. Then, by the previous arguments we would have that $E_n \subset T(E_{n+1})$ for every $n \geq N$. Consider the operator $T_N : \oplus_{k=1}^N F_k \longrightarrow \oplus_{k=1}^N F_k$ defined by $T_N := P_N \circ T$, where $P_N$ is the projection of $\omega = (\oplus_{k=1}^N F_k) \oplus E_N$ onto $\oplus_{k=1}^N F_k$. We now have two cases:
	\begin{enumerate}[label=\arabic*.]
		\item If $T_N$ was invertible, then we would have that $T( (\oplus_{k=1}^N F_n) \oplus E_{N+1}) = \omega$. Hence, for every non-zero vector $x \in F_{N+1} \setminus \{0_{\omega}\}$ there would exist another vector $y \in \oplus_{k=1}^N F_n \oplus E_{N+1}$ fulfilling that $Tx = Ty$ while $x \neq y$.
		
		\item If $T_N$ was not invertible, then there would exist a non-zero vector $x \in \oplus_{k=1}^N F_k$ such that $T_Nx=0_{\omega}$. This would mean that $Tx \in E_N$, so that we could find another vector $y \in E_{N+1}$ fulfilling that $Tx = Ty$ while $x \neq y$. 
	\end{enumerate}
	In both cases we get a contradiction with the injectivity of $T$. 
\end{proof}

By Theorem~\ref{The:intersect-n>=1}, given any operator $T$ on $\omega$ and a closed subspace of finite codimension $E \subset \omega$, then the closed subspaces
\[
M_1 := \bigcap_{n \geq 1}T^{-n}(E) \quad \text{ and } \quad M_0 := \bigcap_{n \geq 0}T^{-n}(E) \subset E
\]
are invariant under the action of $T$. It follows from Theorem~\ref{The:intersect-n>=1} that $M_1$ is never reduced to $\{0_{\omega}\}$, but it can clearly be the whole space even if $E$ is different to the whole space $\omega$: consider the null operator on $\omega$ for any $E$, or the forward shift on $\omega$ for the subspace of finite codimension $E = \ker(p_1) = \{ x = ([x]_j)_{j \geq 1} \in \omega : [x]_1 = 0 \}$.

On the other hand, $M_0$ is never equal to $\omega$ unless we consider $E=\omega$, but it can be equal to $\{0_{\omega}\}$ if $T$ is not injective: the backward shift on $\omega$ does not admit a non-trivial closed invariant subspace included in $E = \ker(p_1)$. However, if $T$ is injective and $E$ is different to $\omega$, then $M_0$ is a non-trivial closed invariant subspace for $T$. This gives us the following consequence:

\begin{corollary}
	Let $T$ be an operator on $\omega$. If $T$ is injective and $E$ is a closed subspace of finite codimension in $\omega$, then the operator $T$ admits a non-trivial closed invariant subspace included in $E$.
\end{corollary}

Since every non-injective operator clearly admits a non-trivial closed invariant subspace (consider its kernel), we can also deduce from Theorem~\ref{The:intersect-n>=1} the existence of non-trivial closed invariant subspaces for any operator $T$ on $\omega$. This was already remarked by K\"orber in~\cite{Korber1969_MathAnn_die-invarianten} for operators on $\omega$, and later by the second author of this paper in~\cite{Menet2018_AdvM_invariant} for $X\oplus\omega$ where $X$ is a Fr\'echet space with continuous norm.

Coming back to our expansivity problem (Question~\ref{Ques:expansive}), using Theorem~\ref{The:intersect-n>=1} we can already show the absence of positively expansive operators on $\omega$:

\begin{theorem}\label{The:no.pos-exp.omega}
	There is no (metrically) positively expansive operator on $\omega$.
\end{theorem}
\begin{proof}
	Let $T$ be an operator on $\omega$. Given $k \in \NN$, let
	\[
	E_k := \ker(p_k) = \{ x = ([x]_j)_{j \geq 1} \in \omega : [x]_j = 0 \text{ for } 1 \leq j \leq k \}.
	\]
	It follows from Theorem~\ref{The:intersect-n>=1} applied to the finite-codimensional subspace $E_k$ that there exists some $x \in \omega \setminus \{0_{\omega}\}$ such that $p_{k}(T^nx) = 0$ for all $n \in \NN$. By considering an enough small multiple of $x$ we can even assume that $p_{k}(x) < \frac{1}{2^{k}}$. Since $(p_l)_{l \geq 1}$ given by \eqref{eq:omega.seminorms} is an increasing sequence of seminorms, we have for the metric $d$ given by \eqref{eq:metric} and for all $n \in \NN_0$ that
	\[
	d(T^nx,0_{\omega}) = \sum_{l=1}^{\infty} \frac{\min\{ 1 , p_l(T^nx) \}}{2^{l}} \leq \frac{1}{2^{k}} \sum_{l=1}^{k} \frac{1}{2^l} + \sum_{l=k+1}^{\infty} \frac{1}{2^l} < \frac{1}{2^{k-1}}.
	\]
	Thus, for each $k \in \NN$ there exists some $x \in \omega \setminus \{0_{\omega}\}$ such that $d(T^nx,0_{\omega}) < 2^{-k+1}$ for every $n \in \NN_0$, which shows that $T$ is not positively expansive with respect to $d$.
	
	We also notice that $T$ is not positively expansive with respect to another metric on $\omega$. Actually, if $d'$ is another metric inducing the topology of $\omega$, the equivalence between $d$ and $d'$ implies that for every $c>0$ there exists some $\eps_c>0$ such that we have the open balls inclusion
	\[
	B_{d}(0_{\omega},\eps_c) \subset B_{d'}(0_{\omega},c).
	\]
	Since for every $c>0$ there exists some $x \in \omega \setminus \{0_{\omega}\}$ such that $d(T^nx,0_{\omega})<\eps_c$ for every $n \in \NN_0$, it follows that $d'(T^nx,0_{\omega})<c$ for every $n \in \NN_0$. Therefore we deduce that $T$ is not positively expansive with respect to $d'$ and thus that $\omega$ does not admit positively expansive operators (in the metric sense).
\end{proof}

In order to show the absence of expansive operators on $\omega$ we are going to adapt the statement and proof of Theorem~\ref{The:intersect-n>=1} to take into account the sets $T^{-n}(E)$ for every integer $n \in \ZZ$. Here we have the statement:

\begin{theorem}\label{The:intersect-|n|>=L}
	For every invertible operator $T$ on $\omega$ and every closed subspace of finite codimension $E \subset \omega$ there exists some $L \geq 0$ fulfilling that 
	\[
	\bigcap_{|n| \geq L}T^{-n}(E) \neq \{0_{\omega}\}.
	\]
\end{theorem}
\begin{proof}
	Let $T$ be an invertible operator on $\omega$ and let $E$ be a closed subspace of finite codimension in $\omega$. We let $E_0 := \omega$ and
	\[
	E_n := E \cap T^{-1}(E) \cap \cdots \cap T^{-(n-1)}(E) \quad \text{ for all } n \geq 1,
	\]
	so that $(E_n)_{n \geq 1}$ is indeed a non-increasing sequence of closed subspaces of finite codimension. Using Proposition~\ref{Pro:F_n+E_n} we can find a sequence of finite-dimensional subspaces $(F_n)_{n \geq 1}$ such that:
	\begin{enumerate}[label=$\bullet$]
		\item $\omega=F_1\oplus E_1$ and $E_n=F_{n+1}\oplus E_{n+1}$ for every $n \geq 1$;
		
		\item the sequence $(\val(F_n))_{n \geq 1}$ tends to infinity and hence every series $\sum_{n=1}^{\infty} f_n$ is convergent whenever $(f_n)_{n \geq 1}$ is a sequence with $f_n \in F_n$ for each $n \geq 1$;
		
		\item for every $x \in \omega$ there exist a unique sequence $(f_n)_{n \geq 1}$, with $f_n\in F_n$ for each $n \geq 1$, and a unique vector $e \in \bigcap_{n \geq 1} E_n$ such that $x = \sum_{n=1}^{\infty} f_n + e$.\\
	\end{enumerate}
	
	Now, consider the subspace
	\[
	M := \bigcap_{n \geq 1} E_n = \bigcap_{n \geq 0}T^{-n}(E).
	\]
	We first remark that, if $M$ is a subspace of finite codimension in $\omega$, then by applying Theorem~\ref{The:intersect-n>=1} and Remark~\ref{Rem:injectivity} to the inverse (and hence injective) operator $T^{-1}$ and to the subspace $M$ we would have that
	\[
	\bigcap_{n\in \ZZ}T^{-n}(E) =\bigcap_{n \geq 0}T^n(M) \neq \{0_{\omega}\}.
	\]
	Therefore we would get the desired result for $L=0$.\\
	
	Thus, let us assume that $M$ is not of finite codimension in $\omega$. In particular, as in the proof of Theorem~\ref{The:intersect-n>=1}, this means that $E_{n+1}$ is a strict subspace of $E_n$ and hence that $F_{n+1}$ is not reduced to $\{0_{\omega}\}$ for all $n \geq 0$. Moreover, since $T$ is invertible and hence injective, again with the same arguments as in Theorem~\ref{The:intersect-n>=1} we can deduce \eqref{eq:dimF_N} and \eqref{eq:x=T(sum)+e}, that is: there exists some $N \geq 1$ such that
	\[
	\dim(F_n) = \dim(F_N) \quad \text{ for all } n \geq N,
	\]
	and given any $n \geq N$ and any $x \in E_n$ there exist a sequence $(f_k)_{k \geq n+2}$, with $f_k \in F_k$ for each $k \geq n+2$, and a vector $m \in M$ such that
	\[
	x = T\left( \sum_{k=n+2}^{\infty}f_k\right ) + m.
	\]
	In particular, from \eqref{eq:x=T(sum)+e} we have that for every $n \geq N$ and every $x\in E_n$ there exist vectors $y \in E_{n+1} \subset E_N$ and $m \in M$ such that
	\begin{equation}\label{eq:Ty+m}
		x = Ty + m.
	\end{equation}
	In addition, using the previous fact multiple times together with the inclusion $T(M) \subset M$ we get that, for every $n \geq N$, every $x\in E_n$ and every $k \geq 1$ there exist vectors $y \in E_{n+k} \subset E_N$ and $m \in M$ such that
	\begin{equation}\label{eq:T^ky+m}
		x=T^ky+m.
	\end{equation}
	These facts \eqref{eq:Ty+m} and \eqref{eq:T^ky+m} will be important to conclude the proof.\\
	
	For each $n \geq 1$ we now let $M_n := M \cap T(M) \cap \cdots \cap T^{n}(M)$, which fulfills that $M_n = T^n(M)$ because $T(M) \subset M$. In particular, we notice that $T(M_n)=M_{n+1}$ and that $M_n$ is of finite codimension in $M$ for every $n \geq 1$ since
	\[
	M_n = T^n(M) = T^n \left(\bigcap_{k \geq 0} T^{-k}(E) \right) = T^n(E_n) \cap M = (T^{-n})^{-1}(E_n) \cap M,
	\]
	and every preimage of a finite codimensional subspace by an operator is again a finite codimensional subspace. We can therefore use Proposition~\ref{Pro:F_n+E_n} for $M$ with the non-increasing sequence of subspaces $(M_n)_{n \geq 1}$ in order to get a family $(G_{n})_{n \geq 1}$ of finite-dimensional subspaces such that:
	\begin{enumerate}[label=$\bullet$]
		\item $M = G_1 \oplus M_1$ and $M_n = G_{n+1} \oplus M_{n+1}$ for every $n \geq 1$;
		
		\item the sequence $(\val(G_n))_{n \geq 1}$ tends to infinity and hence every series $\sum_{n=1}^{\infty} g_n$ is convergent whenever $(g_n)_{n \geq 1}$ is a sequence with $g_n \in G_n$ for each $n \geq 1$;
		
		\item for every $x \in M$ there exist a unique sequence $(g_n)_{n \geq 1}$, with $g_n \in G_n$ for each $n \geq 1$, and a unique vector $z \in \bigcap_{n \geq 1} M_n$ such that $x = \sum_{n=1}^{\infty} g_n + z$.\\
	\end{enumerate}
	
	As in the case in which $M$ was finite codimensional in $\omega$, if we had that
	\[
	\bigcap_{n \geq 0} T^n(M) \neq \{0_{\omega}\}
	\]
	then we would conclude that
	\[
	\bigcap_{n\in \ZZ} T^{-n}(E) = \bigcap_{n \geq 0} T^n(M) \neq \{0_{\omega}\},
	\]
	getting the desired result for $L=0$. Therefore, assume that $\bigcap_{n \geq 0} T^n(M) = \{0_{\omega}\}$ or, equivalently, that $\bigcap_{n \geq 1} M_n = \{0_{\omega}\}$. A consequence of this assumption is that:
	\begin{enumerate}[label=$\bullet$]
		\item the series $\sum_{n=1}^{\infty} m_n$ is convergent for every sequence $(m_n)_{n \geq 1}$ with $m_n\in M_n$ for each $n \geq 1$.
	\end{enumerate}
	Indeed, since $\bigcap_{n \geq 1} M_n = \{0_{\omega}\}$ we have that each element $m_n \in M_n$ can be written by a series $\sum_{k=n+1}^{\infty} g_k$ with $g_k \in G_k$ for each $k \geq n+1$, and since the sequence of valuations $(\val(G_k))_{k \geq 1}$ tends to infinity we deduce that the sequence of valuations $(\val(M_k))_{k \geq 1}$ also tends to infinity, implying the stated consequence.\\
	
	We now claim that there exist some vector $x \in E_N \setminus M$ and some positive integer $L \geq 1$ fulfilling that $T^L x\in M$. Actually, consider $L \geq 1$ such that
	\[
	\dim(\oplus_{k=1}^N F_k) < \dim(\oplus_{k=N+1}^{N+L} F_k),
	\]
	and let $T_{N,L} : \oplus_{k=N+1}^{N+L} F_k \longrightarrow \oplus_{k=1}^N F_k$ be the operator $T_{N,L} := P_N \circ T^L$, where $P_N$ is the projection of $\omega = (\oplus_{k=1}^N F_k) \oplus E_N$ onto $\oplus_{k=1}^N F_k$. The inequality between the dimensions shows that there exists some vector $z \neq 0_{\omega}$ in the kernel of the finite-dimensional operator $T_{N,L}$, so that
	\[
	z \in \oplus_{k=N+1}^{N+L} F_k \subset E_N \setminus E_{N+L} \quad \text{ and } \quad T^Lz \in E_N.
	\]
	Using~\eqref{eq:T^ky+m} we deduce that there exist vectors $y \in E_{N+L} \subset E_N$ and $m \in M$ with
	\[
	T^Lz=T^Ly+m \quad \text{ and hence } \quad T^{L}(z-y) \in M.
	\]
	Thus, the vector $x = z-y$ belongs to $E_N\setminus M$ while $T^Lx \in M$.\\
	
	To conclude the proof we finally claim that for such vector $x \in E_N \setminus M$ with $T^Lx \in M$, there exists a sequence $(m_n)_{n \geq 0} \subset M$ such that
	\[
	0_{\omega} \neq x-\sum_{n=0}^{\infty}T^{n}m_n \in \bigcap_{|n| \geq L}T^{-n}(E).
	\]
	This sequence $(m_n)_{n \geq 0}$ is obtained recursively. First, using~\eqref{eq:Ty+m} there exist vectors $y_0 \in E_N$ and $m_0 \in M$ such that $x = Ty_0 + m_0$. We then let $x_0 := x-m_0$, so that
	\[
	x_0 \in E_N \subset E \quad \text{ and } \quad T^{-1} x_0 = y_0 \in E_N \subset E.
	\]
	Assume that for some $k \geq 0$ we have already fixed $m_0,\dots,m_k \in M$ such that
	\[
	x_k = x-\sum_{n=0}^k T^{n}m_n
	\]
	satisfies $T^{-l}x_k \in E_N \subset E$ for every $0 \leq l \leq k+1$. Thus, since $T^{-(k+1)}x_k\in E_N$, it follows from~\eqref{eq:Ty+m} that there exist vectors $y_{k+1} \in E_N$ and $m_{k+1} \in M$ such that $T^{-(k+1)}x_k=T y_{k+1}+m_{k+1}$. Letting $x_{k+1} := x_{k}-T^{k+1}m_{k+1} = x-\sum_{n=0}^{k+1}T^{n}m_n$ we have that
	\[
	T^{-l}x_{k+1}=T^{-l}x_k -T^{k+1-l}m_{k+1}\in E_N \subset E \quad \text{  for every } 0 \leq l \leq k+1,
	\]
	but also that
	\[
	T^{-(k+2)} x_{k+1} = T^{-1}(T^{-(k+1)}x_k-m_{k+1}) = y_{k+1} \in E_N \subset E.
	\]
	Recursively we obtain $(x_k)_{k \geq 0}$ and $(m_k)_{k \geq 0}$ fulfilling that $x_k \in \bigcap_{l=0}^{k+1} T^l(E)$ for every $k \geq 0$. In particular, the vector
	\[
	x^* := \lim_{k \to \infty} x_k = x-\sum_{n=0}^{\infty}T^n m_n,
	\]
	which is well-defined since $T^nm_n \in M_n$ for each $n \geq 1$ and the series $\sum_{n=0}^{\infty}T^n m_n$ is thus convergent, fulfills that $x^* \in \bigcap_{l \geq 0}T^{l}(E)$. Moreover, notice that
	\[
	T^Lx^* = T^Lx - \sum_{n=0}^{\infty} T^{L+n}m_n \in M = \bigcap_{l \geq 0} T^{-l}(E),
	\]
	and hence $x^* \in \bigcap_{l \geq L}T^{-l}(E)$. Since $x \notin M$ and $\sum_{n=0}^{\infty}T^{n}m_n \in M$, we finally have that $0_{\omega} \neq x^* \in \bigcap_{|n| \geq L}T^{-n}(E)$, concluding the proof.
\end{proof}

\begin{remark}
	Given an invertible operator $T$ on $\omega$ and a closed subspace of finite codimension $E \subset \omega$, it is possible that for every $L \geq 0$ the space $\bigcap_{|n| \geq L}T^{-n}(E)$ is finite-dimensional. For instance, if we denote by $(e_n)_{n \geq 1}$ the canonical basis of $\omega$ and by $(e_n^*)_{n \geq 1}$ the corresponding coordinate functionals, then we can consider the operator $T$ given by $Te_{2n} := e_{2(n-1)}$ for each $n \geq 2$, $Te_2 := e_1$, and $Te_{2n+1} := e_{2n+3}$ for each $n \geq 0$. Thus, for $E := \ker(e_1^*) = \ker(p_1)$, given any $l \geq 1$ we have that
	\[
	T^l(E)=\ker(e^*_{2l+1}) \quad \text{ and } \quad T^{-l}(E)=\ker(e^*_{2l})
	\]
	so that $\bigcap_{|n| \geq L} T^{-n}(E)$ is finite-dimensional for every $L \geq 1$. This example also shows that we can have the equality $\bigcap_{n\in \ZZ}T^{-n} (E) = \{0_{\omega}\}$. Note that, if we reorder the indices of the basis $(e_n)_{n \geq 1}$ using the set $\ZZ$ appropriately, then the operator $T$ coincides with the classical bilateral shift on $\KK^{\ZZ}$.
\end{remark}

We end up this section by showing how we can deduce from Theorem~\ref{The:intersect-|n|>=L} the absence of expansive operators on $\omega$.

\begin{theorem}
	There is no (metrically) expansive operator on $\omega$.
\end{theorem}
\begin{proof}
	Let $T$ be an invertible operator on $\omega$. Given $k \in \NN$, let
	\[
	E_k := \ker(p_k) = \{ x = ([x]_j)_{j \geq 1} \in \omega : [x]_j = 0 \text{ for } 1 \leq j \leq k \}.
	\]
	It follows from Theorem~\ref{The:intersect-|n|>=L} applied to the finite-codimensional subspace $E_k$ that there exist some integer $L \geq 0$ and some vector $x \in \omega \setminus \{0_{\omega}\}$ such that $p_k(T^nx) = 0$ for all $n \in \ZZ$ with $|n| \geq L$. By considering an enough small multiple of $x$ we can assume that $p_k(T^nx) < \frac{1}{2^{k}}$ for all $n \in \ZZ$ with $|n|<L$. Since $(p_l)_{l \geq 1}$ given by \eqref{eq:omega.seminorms} is an increasing sequence of seminorms, we have for the metric $d$ given by \eqref{eq:metric} and for every integer $n \in \ZZ$ that
	\[
	d(T^nx,0_{\omega}) = \sum_{l=1}^{\infty} \frac{\min\{ 1 , p_l(T^nx) \}}{2^{l}} \leq \frac{1}{2^{k}} \sum_{l=1}^{k} \frac{1}{2^l} + \sum_{l=k+1}^{\infty} \frac{1}{2^l} < \frac{1}{2^{k-1}}.
	\]
	Thus, for each $k \in \NN$ there exists some $x \in \omega \setminus \{0_{\omega}\}$ such that $d(T^nx,0_{\omega}) < 2^{-k+1}$ for every $n \in \ZZ$. Arguing exactly as in the proof of Theorem~\ref{The:no.pos-exp.omega} we can deduce that $T$ is not an expansive operator with respect to $d$, and that $\omega$ does not admit expansive operators in the metric sense.
\end{proof}

\section{Expansive operators on Fr\'echet spaces without continuous norm}

The objective of this section is proving that there are Fr\'echet spaces without continuous norm that admit (positively) expansive operators. We first consider spaces of the type $X\oplus\omega$, where $X$ is an infinite-dimensional Banach space, and we identify sufficient conditions on $X$ for the existence of this kind of operators on the direct sum space $X\oplus\omega$.

\begin{theorem}\label{The:X+omega.pos-exp}
	If a Banach space $X$ admits a basic sequence $(u_k)_{k \geq 1}$ such that:
	\begin{enumerate}[label=$\bullet$]
		\item $\inf_{k \geq 1} \|u_k\|>0$;
		
		\item the closed subspace $M : =\overline{\lspan}\{ u_k : k \geq 1 \}$ is topologically complemented in $X$;
		
		\item and the forward shift defined along $(u_k)_{k \geq 1}$ is continuous on $M$;
	\end{enumerate}
	then the Fr\'echet space $X\oplus\omega$ admits a (metrically) positively expansive operator.
\end{theorem}
\begin{proof}
	Let $Y$ be a subspace of $X$ such that $X = Y \oplus M$. We consider the projections $P_Y : X\oplus\omega \longrightarrow Y$, $P_M : X\oplus\omega \longrightarrow M$ and  $P_{\omega} : X\oplus\omega \longrightarrow \omega$ such that for any vector $z \in X\oplus\omega$ we can write $z = P_Yz + P_Mz + P_{\omega}z$. We now define an operator $T$ on $X\oplus\omega$ by letting
	\[
	Tz := 2P_Yz + 2F_u(P_Mz) + [P_{\omega}z]_1 u_1 + B_e(P_{\omega}z) \quad \text{ for each } z \in X\oplus\omega,
	\] 
	where $F_u : M \longrightarrow M$ is the forward shift along $(u_k)_{k \geq 1}$ and $B_e : \omega \longrightarrow \omega$ is the backward shift on $\omega$ along its canonical basis $(e_k)_{k \geq 1}$. Then, the linear map $T$ is continuous since the projections $P_Y$, $P_M$ and $P_{\omega}$ are continuous, since the forward shift $F_u$ is continuous on $M$ by assumption, and since the backward shift $B_e$ is continuous on $\omega$. It remains to show that $T$ is positively expansive. To check this fact notice that, since $P_Y$ and $P_M$ are continuous, there exist a continuous seminorm $q$ on $X\oplus\omega$ and a constant $C>0$ such that
	\[
	\|P_Y z\|_X \leq C q(z) \quad \text{ and } \quad \|P_M z\|_X \leq C q(z) \quad \text{ for every } z \in X\oplus\omega.
	\]
	Thus, given any non-zero vector $z \in X\oplus\omega$ we have several  non-mutually exclusive possibilities:
	\begin{enumerate}[label=$\bullet$]
		\item If $P_Yz \neq 0_Y$, then we get that
		\[
		\|P_Y(T^nz)\|_X = 2^n \|P_Yz\|_X \quad \text{ for every } n \geq 1,
		\]
		which implies that $q(T^nz) \geq \frac{1}{C}\|P_Y(T^nz)\|_X \geq 2^n \frac{\|P_Yz\|_X}{C}$. We can then conclude that $q(T^nz) \to \infty$ when $n \to \infty$
		
		\item If $P_Mz \neq 0_M$, then we get $P_Mz = \sum_{k=1}^{\infty} z_k u_k$ with some $z_{k_0} \neq 0$. Since $(u_k)_{k \geq 1}$ is a basic sequence, each coordinate functional $u_k^* : M \longrightarrow \KK$ satisfies that
		\[
		\|u_k^*\| \leq \frac{2K}{\|u_k\|_X},
		\]
		where $K$ is the basis constant of $(u_k)_k$; see for instance \cite[Vol.~I, Section~1.b.]{LindestraussTza1977_book_classical}. We can then deduce that 
		\begin{align*}
			2^n |z_{k_0}| = \left|u_{k_0+n}^*(P_M(T^n z))\right| &\leq \|u_{k_0+n}^*\| \cdot \|P_M(T^n z)\|_X \\[5pt]
			& \leq \frac{2KC}{\inf_{k \geq 1} \|u_k\|_X} \cdot q(T^n z),
		\end{align*}
		which implies that $q(T^nz) \to \infty$ when $n \to \infty$.
		
		\item If $P_{\omega}z \neq 0_{\omega}$, then we get $P_{\omega}z = \sum_{k=1}^{\infty} w_ke_k$ with some $w_{k_0} \neq 0$. In this case it suffices to remark that $T^{k_0}z$ will satisfy $P_M(T^{k_0}z) \neq 0_M$, and by the previous case we deduce that $q(T^nz) \to \infty$ when $n \to \infty$.
	\end{enumerate}
	In conclusion, we get that $q(T^nz) \to \infty$ for every $z \in X\oplus\omega \setminus \{0_{X \oplus \omega}\}$ so that for any translation-invariant metric $d$ inducing the topology of $X\oplus\omega$ there must exists some $c>0$ such that $d(T^nz,0_{X \oplus \omega}) \geq c$ for all sufficiently big $n \in \NN$.
\end{proof}

If we additionally assume in Theorem~\ref{The:X+omega.pos-exp} that the backward shift along $(u_k)_{k \geq 1}$ is continuous, then we even get an expansive operator.

\begin{theorem}\label{The:X+omega.exp}
	If a Banach space $X$ admits a basic sequence $(u_k)_{k \geq 1}$ such that:
	\begin{enumerate}[label=$\bullet$]
		\item $\inf_{k \geq 1}\|u_k\|>0$;
		
		\item the closed subspace $M := \overline{\lspan}\{ u_k : k \geq 1 \}$ is topologically complemented in $X$;
		
		\item the backward and forward shifts along $(u_k)_{k \geq 1}$ are continuous on $M$;
	\end{enumerate}
	then $X\oplus\omega$ admits an invertible positively expansive (and hence expansive) operator.
\end{theorem}
\begin{proof}
	With the continuity of the backward shift assumption we have that the operator $T$ defined in Theorem~\ref{The:X+omega.pos-exp} is invertible. In fact, its inverse $T^{-1}$ is given by
	\[
	T^{-1}z = \frac{1}{2}P_Yz + \frac{1}{2}B_u(P_Mz) + [P_Mz]_1 e_1 + F_e(P_{\omega}z) \quad \text{ for each } z \in X\oplus\omega,
	\] 
	where $B_u : M \longrightarrow M$ is the backward shift along $(u_k)_{k \geq 1}$ and $F_e : \omega \longrightarrow \omega$ is the forward shift on $\omega$ along its canonical basis $(e_k)_{k \geq 1}$.
\end{proof}

We remark that Theorem~\ref{The:X+omega.exp} applies to many classical Banach spaces:

\begin{corollary}
	The Fr\'echet space $X\oplus\omega$ admits an invertible positively expansive operator and thus an expansive operator in any of the following cases: 
	\begin{enumerate}[label=$\bullet$]
		\item $X=\ell^p$ with $p \in [1,\infty[$;
		
		\item $X=c_0$;
		
		\item $X$ is an infinite-dimensional Hilbert space;
		
		\item $X$ is a non-reflexive Banach space admitting an unconditional basis.
	\end{enumerate}
\end{corollary}
\begin{proof}
	It suffices to remark that these spaces admit a basic sequence satisfying the assumptions of Theorem~\ref{The:X+omega.exp}. If $X=\ell^p$ for $p \in [1,\infty[$ or $c_0$, we can let $(u_k)_{k \geq 1}$ be its canonical basis. If $X$ is an infinite-dimensional Hilbert space, it suffices to consider any orthonormal sequence $(u_k)_{k \geq 1}$ in $X$. Finally, if $X$ is a non-reflexive Banach space admitting an unconditional basis, using James~\cite{James1950_AnnalsMath_bases} we know that $X$ contains a complemented copy of $\ell^1$ or a complemented copy of $c_0$.
\end{proof}

We conclude the paper by exhibiting some Fréchet spaces without continuous norm and formed as $\omega$-products of infinite-dimensional Banach spaces that support expansive operators. We will consider the Fr\'echet spaces $\omega(\ell^p)$ and $\omega(c_0)$ where, for a Banach space $X$, the space $\omega(X)$ denotes the space $X^{\NN}$ endowed with the sequence of seminorms $(p_k)_{k \geq 1}$ defined for each $k \in \NN$ as
\[
p_k(x) := \max_{1 \leq j \leq k} \| x^{(j)} \|_X \quad \text{ for each } x = (x^{(j)})_{j \geq 1} \in \omega(X).
\]
For these Fr\'echet spaces we have the following result: 

\begin{prop}
	The spaces $\omega(\ell^p)$ with $p\in [1,\infty[$ and $\omega(c_0)$ admit an invertible positively expansive operator and thus an expansive operator.
\end{prop}
\begin{proof}
	Let $X=\ell^p$ with $p\in [1,\infty[$ or $c_0$ and denote by $(e_n)_{n \geq 1}$ the canonical basis of such space. Let $(A_j)_{j \geq 1}$ be a partition of $\NN$ where each set $A_j$ is infinite. For each $j \geq 1$, we denote by $(n_{j,k})_{k \geq 1}$ the increasing enumeration of $A_j$ and we then consider the operator $F_A$ defined on $X$ by 
	\[
	F_A\left(\sum_{j=1}^{\infty}\sum_{k=1}^{\infty}x_{n_{j,k}} e_{n_{j,k}}\right)=\sum_{j=1}^{\infty}\sum_{k=1}^{\infty}x_{n_{j,k}} e_{n_{j+1,k}}.
	\]
	In other words, we identify $X$ with an infinity copy of itself and $F_A$ is the forward shift along these copies. In the same way, we define the operator $B_A$ on $X$ by
	\[
	B_A\left(\sum_{j=1}^{\infty}\sum_{k=1}^{\infty}x_{n_{j,k}} e_{n_{j,k}}\right)=\sum_{j=2}^{\infty}\sum_{k=1}^{\infty}x_{n_{j,k}} e_{n_{j-1,k}}.
	\]
	We now consider the operator $T : \omega(X) \longrightarrow \omega(X)$ given by
	\[
	Tx := \left(2F_A(x^{(1)})+\sum_{k=1}^{\infty} [x^{(2)}]_k e_{n_{1,k}}, x^{(3)}, x^{(4)}, x^{(5)}, \dots \right),
	\]
	for each $x = (x^{(j)})_{j \geq 1} \in \omega(X)$. This operator is continuous on $\omega(X)$ and is even invertible. In fact, it is not hard to check that
	\[
	T^{-1}x = \left( \frac{1}{2}B_A(x^{(1)}), \sum_{k=1}^{\infty} [x^{(1)}]_{n_{1,k}} e_k, x^{(2)}, x^{(3)}, x^{(4)}, \dots \right),
	\]
	for each $x = (x^{(j)})_{j \geq 1} \in \omega(X)$. Moreover the operator $T$ is positively expansive since for every vector $x = (x^{(j)})_{j \geq 1} \in \omega(X) \setminus \{0_{\omega(X)}\}$ we get that $p_1(T^n x)$ tends to infinity. Indeed, if $x = (x^{(j)})_{j \geq 1}$ is a non-zero vector in $\omega(X)$, there exists $j_0 \geq 1$ such that $x^{(j_0)}$ is different to $0_X$ and we then have $p_1(T^{j_0-1}x) \neq 0$. Moreover, for every $n \geq 0$ and every $x \in \omega(X)$ we have, by the definition of $T$, that
	\[
	p_1(T^n x) \geq 2^np_1(x).
	\]
	We then get the desired conclusion.
\end{proof}

The results obtained in this paper lead to the following open problem:

\begin{pbm}
	Characterize the Banach (or Fr\'echet) spaces $X$ for which the direct sum $X\oplus\omega$ admits a (positively) expansive operator in the metric sense.
\end{pbm}

We notice that $X$ must be infinite-dimensional since otherwise $X \oplus \omega$ would be isomorphic to $\omega$.

\vspace{1cm}

{

}

\vspace{1cm}

{\small
\noindent\textsc{Antoni L\'opez-Mart\'inez}: Universitat Polit\`ecnica de Val\`encia, Institut Universitari de Matem\`atica Pura i Aplicada, Edifici 8E, 4a planta, 46022 Val\`encia, Spain.\\ e-mail: alopezmartinez@mat.upv.es\\

\noindent\textsc{Quentin Menet}: Universit\'e de Mons, Service d'Analyse fonctionnelle, D\'epartement de Math\'ematique, Place du Parc 20, 7000 Mons, Belgium. \\ e-mail: quentin.menet@umons.ac.be
}

\end{document}